\documentclass{amsart}
\usepackage{amssymb,mathtools}
\usepackage[british]{babel}
\usepackage{enumitem}
\usepackage[latin1]{inputenc}
\usepackage{url}
\usepackage{hyperref}
\usepackage{tikz-cd}

\newtheorem{theorem}{Theorem}
\numberwithin{theorem}{section}
\newtheorem{corollary}[theorem]{Corollary}
\newtheorem{lemma}[theorem]{Lemma}
\newtheorem{proposition}[theorem]{Proposition}
\theoremstyle{definition}
\newtheorem{definition}[theorem]{Definition}
\newtheorem{remark}[theorem]{Remark}
\newtheorem{example}[theorem]{Example}

\newcommand{\rca}{\mathbf{RCA}}
\newcommand{\aca}{\mathbf{ACA}}
\newcommand{\supp}{\operatorname{supp}}
\newcommand{\otp}{\operatorname{otp}}

\newcommand{\rng}{\operatorname{rng}}

\newcommand{\en}{\operatorname{en}}
\newcommand{\lkb}{<_{\operatorname{KB}}}

\newcommand{\fix}{\operatorname{Fix}}
\newcommand{\npm}{\mathbb N^\infty_{-1}}

\title[Single fixed points of normal functions]{How strong are single fixed points\\ of normal functions?}

\author{Anton Freund}
\address{Fachbereich Mathematik, Technische Universit\"at Darmstadt, Schlossgartenstr.~7, 64289~Darmstadt, Germany}
\email{freund@mathematik.tu-darmstadt.de}

\begin{document}


\begin{abstract}
In a recent paper by M.~Rathjen and the present author it has been shown that the statement ``every normal function has a derivative'' is equivalent to $\Pi^1_1$-bar induction. The equivalence was proved over $\mathbf{ACA_0}$, for a suitable representation of normal functions in terms of dilators. In the present paper we show that the statement ``every normal function has at least one fixed point'' is equivalent to $\Pi^1_1$-induction along the natural numbers.
\end{abstract}

\maketitle

{\let\thefootnote\relax\footnotetext{\copyright~The Association for Symbolic Logic 2020.\\
This is the accepted version of a publication in \emph{The Journal of Symbolic Logic} 85:2 (2020) 709-732.}

\section{Introduction}

Recall that a function from ordinals to ordinals is called normal if it is strictly increasing and continuous at limit stages. More explicitly, $f$ is a normal function if
\begin{enumerate}[label=(\roman*)]
 \item $\alpha<\beta$ implies $f(\alpha)<f(\beta)$ and
 \item we have $f(\lambda)=\sup_{\alpha<\lambda}f(\alpha)$ whenever $\lambda$ is a limit ordinal.
\end{enumerate}
Equivalently, the function $f$ is the unique increasing enumeration of a closed and unbounded (club) class of ordinals. To construct a fixed point of a normal function~$f$ it suffices to consider iterates: Recursively define
\begin{equation*}
 f^0(\alpha)=\alpha,\qquad f^{n+1}(\alpha)=f(f^n(\alpha)).
\end{equation*}
One readily checks that
\begin{equation*}
 f'(0)=\textstyle\sup_{n\in\mathbb N}f^n(0)
\end{equation*}
is the smallest ordinal with $f(f'(0))=f'(0)$ (use continuity at the limit $f'(0)$, except if $f(0)=0=f'(0)$). It is well-known that any normal function $f$ does in fact have a club class of fixed points. The normal function that enumerates this class is called the derivative of $f$ and will be denoted by $f'$.

As shown by M.~Rathjen and the present author~\cite{freund-rathjen_derivatives}, a suitable formalization of the statement that ``every normal function has a derivative" is equivalent to bar induction (also known as transfinite induction) for $\Pi^1_1$-formulas, with $\aca_0$ as base theory. Let us stress that this result is formulated within the framework of second order arithmetic (see~\cite{simpson09} for a comprehensive introduction). It relies on a suitable representation of normal functions, which uses J.-Y.~Girard's~\cite{girard-pi2} notion of dilator and related ideas by P.~Aczel~\cite{aczel-phd,aczel-normal-functors}. Details of the representation have been worked out in~\cite{freund-rathjen_derivatives} and will be recalled in Section~\ref{sect:representation} below. We should point out that the use of dilators does lead to some restrictions: In particular dilators with countable parameters cannot raise infinite cardinalities (cf.~\cite[Remark~2.3.6]{girard-pi2}), which means that a normal function such as $\alpha\mapsto\aleph_\alpha$ is beyond the scope of the present paper. On the other hand, dilators can be used to represent many normal functions that arise in proof theory and computability theory (cf.~\cite{schuette77,marcone-montalban}). They also support a rich theory of general constructions on normal functions. For the rest of this introduction we proceed in an informal manner, ignoring the difference between normal functions and their representations in second order arithmetic. Official versions of our results can be found in the following sections.

In view of the aforementioned result from~\cite{freund-rathjen_derivatives} it is natural to ask: How much induction is needed to ensure that every normal function has at least $\alpha$ fixed points? The proof of Theorem~5.8 from~\cite{freund-rathjen_derivatives} reveals that induction along $\omega\cdot\alpha$ is sufficient. Conversely, the proof of Corollary~3.13 from the same paper shows that $\alpha$ fixed points of a suitable normal function secure induction along~$\alpha$. If $\alpha$ is infinite, then these upper and lower bounds match, since any induction along $\omega\cdot\alpha$ can be reduced to an induction along $\alpha$, with a side induction along $\omega\leq\alpha$. Let us now argue that the answer for any finite $\alpha>0$ coincides with the one for $\alpha=1$: Given a normal function $f$, the idea is to construct the fixed points $f'(n)$ by recursion in the meta theory. Assuming that $f'(n)$ is given, we can consider the normal function $f_n$ with
\begin{equation*}
 f_n(\gamma)=f(f'(n)+1+\gamma).
\end{equation*}
The case of $\alpha=1$ provides the fixed point $f_n'(0)$, which is readily seen to be an upper bound for $f'(n+1)$. Let us indicate in which way this bound secures the value $f'(n+1)$ itself: In Section~4 of~\cite{freund-rathjen_derivatives} (see also Section~\ref{sect:ind-to-wf} of the present paper) it has been shown that notation systems for the fixed points of a given normal function can be constructed in $\rca_0$. The inequality $f'(n+1)\leq f_n'(0)$ corresponds to an embedding, which reduces the claim that the notation system for $f'(n+1)$ is well-founded to the same claim about $f_n'(0)$. The main result of the present paper solves the remaining case, by showing that suitable formalizations of the following statements are equivalent over $\aca_0$ (see Theorem~\ref{thm:main-result} for the precise result):
\begin{enumerate}
\item Every normal function has a fixed point.
\item The principle of $\Pi^1_1$-induction along the natural numbers holds.
\end{enumerate}
Since induction along $\omega\cdot n$ is readily reduced to induction along $\omega$, we can give the following uniform answer to the question from the beginning of the paragraph: The statement that every normal function has at least $\alpha$ fixed points is equivalent to $\Pi^1_1$-induction along $\omega\cdot\alpha$, for any ordinal $\alpha$.

As pointed out by the anonymous referee, statements~(1) and~(2) are also equivalent to a suitable formalization of the following (see~Remark~\ref{rmk:large-fixed-points} for details):
\begin{enumerate}\setcounter{enumi}{2}
\item Every normal function has arbitrarily large fixed points.
\end{enumerate}
More precisely, the formalization of~(3) will assert that any well order can be embedded into some fixed point of a given normal dilator. We emphasize that~(3) does not imply the existence of derivatives. This follows from the aforementioned result of Rathjen and the present author~\cite{freund-rathjen_derivatives}, together with the fact that $\Pi^1_1$-induction along~$\mathbb N$ is strictly weaker than $\Pi^1_1$-bar induction. In~\cite{freund-ordinal-exponentiation} it is shown that the existence of derivatives implies arithmetical comprehension. Hence the equivalence from~\cite{freund-rathjen_derivatives} does already hold over~$\rca_0$. We do not know whether the base theory of the present paper can be weakened as well.

In the rest of this introduction we give an informal argument for the equivalence between (1) and (2). The following sections will show that this argument can be made precise in an appropriate way (similarly to the formalization in~\cite{freund-rathjen_derivatives}). To show that (1) implies (2) we must establish induction for a $\Pi^1_1$-formula~$\varphi$. The Kleene normal form theorem (see~\cite[Lemma~V.1.4]{simpson09}) provides a family of trees $\mathcal T_n$ with
\begin{equation*}
\varphi(n)\leftrightarrow\text{``$\mathcal T_n$ is well-founded"}.
\end{equation*}
The premises of the induction statement for $\varphi$ can then be written as
\begin{gather*}
\text{``$\mathcal T_0$ is well-founded"},\\
\forall_n(\text{``$\mathcal T_n$ is well-founded"}\rightarrow\text{``$\mathcal T_{n+1}$ is well-founded"}).
\end{gather*}
Assume that these statements are witnessed by an ordinal $\alpha_0$ and a function $h$, in the sense that we have
\begin{gather*}
\otp(\mathcal T_0)\leq\alpha_0,\\
\forall_{n,\gamma}(\otp(\mathcal T_n)\leq\gamma\rightarrow\otp(\mathcal T_{n+1})\leq h(n,\gamma)),
\end{gather*}
where $\otp(\mathcal T)$ denotes the order type (or rank) of $\mathcal T$. To avoid the dependency on~$n$ we set $h_0(\gamma)=\sup_{n\in\mathbb N}h(n,\gamma)$. Now consider the normal function $f$ given by
\begin{equation*}
f(\delta)=\alpha_0+1+\sum_{\gamma<\delta}(h_0(\gamma)+1).
\end{equation*}
Note that the infinite sum corresponds to a transfinite recursion, in which the summand $h_0(\gamma)+1$ is added at the successor stage $\gamma+1$. This immediately yields
\begin{equation*}
\gamma+1\leq\delta\quad\Rightarrow\quad h_0(\gamma)+1\leq f(\delta).
\end{equation*}
By construction we have $\otp(\mathcal T_0)+1\leq\alpha_0+1\leq f(f'(0))=f'(0)$. Using the above we also see that $\otp(\mathcal T_n)+1\leq f'(0)$ implies
\begin{equation*}
\otp(\mathcal T_{n+1})+1\leq h(n,\otp(\mathcal T_n))+1\leq h_0(\otp(\mathcal T_n))+1\leq f(f'(0))=f'(0).
\end{equation*}
Using induction on $n$ we obtain
\begin{equation*}
\otp(\mathcal T_n)<f'(0)
\end{equation*}
for all $n\in\mathbb N$, where the reference to $f'(0)$ is secured by statement (1) above. This means that the fixed point $f'(0)$ witnesses the statement
\begin{equation*}
\forall_n\,\text{``$\mathcal T_n$ is well-founded"},
\end{equation*}
which corresponds to the conclusion $\forall_n\varphi(n)$ of induction for $\varphi$. To show that (2) implies (1) we will construct a notation system $\fix(f)$ for the first fixed point of a given normal function~$f$. As indicated above, this can be done in $\rca_0$. In order to establish (1) we must prove that $\fix(f)$ is well-founded. For this purpose we show that $\fix(f)$ is the union of initial segments that correspond to the iterates~$f^n(0)$. The well-foundedness of these initial segments can be established by induction on~$n$, as justified by statement (2).

\section{Representing normal functions by dilators}\label{sect:representation}

As mentioned in the introduction, Girard's~\cite{girard-pi2} notion of dilator and related ideas by Aczel~\cite{aczel-phd,aczel-normal-functors} make it possible to represent normal functions in the setting of second order arithmetic. Details of this representation have been worked out in~\cite[Section~2]{freund-rathjen_derivatives}. In the present section we recall the relevant parts of that paper. Once normal functions have been represented, it will be straightforward to define an appropriate notion of fixed point.

In order to define dilators we consider the category of linear orders, with strictly increasing functions (i.\,e.~embeddings) as morphisms. By the category of natural numbers we mean the full subcategory with the finite orders $n=\{0,\dots,n-1\}$ as objects. This is a small category that is equivalent to the category of finite orders. To witness the equivalence of categories we associate each finite order~$a$ with its increasing enumeration $\en_a:|a|\rightarrow a$. Each embedding $f:a\rightarrow b$ of finite orders is associated with a unique function $|f|:|a|\rightarrow|b|$ that satisfies
\begin{equation*}
f\circ\en_a=\en_b\circ|f|.
\end{equation*}
This turns $|\cdot|$ and $\en$ into a functor and a natural equivalence. We will usually omit the forgetful functor that maps a linear order to its underlying set. In particular~the finite subset functor $[\cdot]^{<\omega}$ with
\begin{align*}
[X]^{<\omega}&=\text{``the set of finite subsets of $X$"},\\
[f]^{<\omega}(a)&=\{f(x)\,|\,x\in a\}
\end{align*}
will also be applied when $X$ is a linear order. Conversely, an element $a\in[X]^{<\omega}$~will then be viewed as a suborder (rather than just a subset) of $X$. It is well-known that finite sets and functions can be coded by natural numbers. Hence the objects in the following definition can be represented by subsets of $\mathbb N$, which allows for a formalization in second order arithmetic.

\begin{definition}[$\rca_0$]\label{def:coded-prae-dilator}
 A coded prae-dilator consists of
 \begin{enumerate}[label=(\roman*)]
  \item a functor $T$ from the category of natural numbers to the category of linear orders, where the field of each order $T(n)$ is a subset of $\mathbb N$, and
  \item a natural transformation $\supp^T:T\Rightarrow[\cdot]^{<\omega}$ such that each $\sigma\in T(n)$ lies in the range of the function $T(\iota_\sigma\circ\en_\sigma):T(|\supp^T_n(\sigma)|)\rightarrow T(n)$, where
  \begin{equation*}
   |\supp^T_n(\sigma)|\xrightarrow{\mathmakebox[2em]{\en_\sigma}}\supp^T_n(\sigma)\xhookrightarrow{\mathmakebox[2em]{\iota_\sigma}}n=\{0,\dots,n-1\}
  \end{equation*}
  compose to the unique embedding with range $\supp^T_n(\sigma)\subseteq n$.
 \end{enumerate}
\end{definition}

In Example~\ref{ex:omega-X} below we describe a coded dilator that represents the normal function $\alpha\mapsto\omega^\alpha$ from ordinal arithmetic. First, however, we want to give a general account of the following crucial observation by Girard: Due to their high uniformity (which is ensured by functoriality and the existence of finite supports), coded prae-dilators can be extended beyond the category of natural numbers. A~concrete construction of this extension in second order arithmetic has been given in~\cite{freund-computable}:

\begin{definition}[$\rca_0$]\label{def:coded-prae-dilator-reconstruct}
 Given a coded prae-dilator $T=(T,\supp^T)$, we set
 \begin{equation*}
 D^T(X)=\{\langle a,\sigma\rangle\,|\,a\in[X]^{<\omega}\text{ and }\sigma\in T(|a|)\text{ and }\supp^T_{|a|}(\sigma)=|a|\}
 \end{equation*}
 for each order~$X$. To define a binary relation on $D^T(X)$ we stipulate
 \begin{equation*}
 \langle a,\sigma\rangle<_{D^T(X)}\langle b,\tau\rangle\quad\Leftrightarrow\quad T(|\iota_a^{a\cup b}|)(\sigma)<_{T(|a\cup b|)}T(|\iota_b^{a\cup b}|)(\tau),
\end{equation*}
where $\iota_a^{a\cup b}$ and $\iota_b^{a\cup b}$ denote the inclusion maps from $a$ resp.~$b$ into $a\cup b$.
\end{definition}

In the proof of~\cite[Lemma~2.4]{freund-computable}, the following result is established in a stronger base theory. It is straightforward to see that $\rca_0$ supports the relevant argument.

\begin{lemma}[$\rca_0$]
 Consider a coded prae-dilator $T$. If $X$ is a linear order, then so is $D^T(X)=(D^T(X),<_{D^T(X)})$.
\end{lemma}

Having extended prae-dilators to arbitrary linear orders, we can now consider the preservation of well-foundedness (note that the two obvious definitions of well-order are equivalent over~$\rca_0$, see e.\,g.~\cite[Lemma~2.3.12]{freund-thesis}):

\begin{definition}[$\rca_0$]\label{def:coded-dilator}
 A coded prae-dilator~$T$ is called a coded dilator if $D^T(X)$ is well-founded for every well-order $X$ on a subset of~$\mathbb N$.
\end{definition}

Let us discuss how Definition~\ref{def:coded-prae-dilator-reconstruct} allows us to represent arbitrary prae-dilators: Working in a suitable base theory, it is natural to define a class-sized prae-dilator as a functor from linear orders to linear orders, together with a natural transformation as in part~(ii) of Definition~\ref{def:coded-prae-dilator} (cf.~\cite[Definition~1.1]{freund-computable} for full details). If $T$ is a class-sized prae-dilator with countable fields $T_n\subseteq\mathbb N$, then its restriction $T\!\restriction\!\mathbb N$ to the category of natural numbers is a coded prae-dilator. Conversely, Proposition~2.5 of~\cite{freund-computable} shows that we get isomorphisms $D^{T\restriction\mathbb N}(X)\cong T(X)$ by stipulating
\begin{equation*}
 D^{T\restriction\mathbb N}(X)\ni\langle a,\sigma\rangle\mapsto T(\iota_a\circ\en_a)(\sigma)\in T(X),
\end{equation*}
where $\iota_a:a\hookrightarrow X$ is the inclusion. The condition $\supp^T_{|a|}(\sigma)=|a|$ from the definition of $D^T(X)$ is crucial for injectivity, as will become clear in Example~\ref{ex:omega-X}. If we extend Definition~\ref{def:coded-prae-dilator-reconstruct} by the clauses
\begin{align*}
 D^T(f)(\langle a,\sigma\rangle)&=\langle [f]^{<\omega}(a),\sigma\rangle,\\
 \supp^{D^T}_X(\langle a,\sigma\rangle)&=a,
\end{align*}
then $D^{T\restriction\mathbb N}$ becomes a class-sized prae-dilator that is isomorphic to~$T$. Concerning the preservation of well-foundedness, we point out that $T(X)$ is well-founded for every well-order~$X$ if the same holds for every countable well-order with field \mbox{$X\subseteq\mathbb N$} (see~\cite[Theorem~2.1.15]{girard-pi2}). Altogether this means that coded \mbox{(prae-)} dilators correspond to class-sized \mbox{(prae-)} dilators that map finite orders to at most countable ones. Only coded (prae-) dilators will play an official role in the present paper. For the sake of readability we will often omit the specification ``coded''.

Let us also point out that our definition of dilators is equivalent to the original one by Girard: In~\cite[Remark~2.2.2]{freund-thesis} it has been verified that an endofunctor~$T$ of linear orders preserves direct limits and pull-backs if, and only if, there is an (automatically unique and hence natural) transformation $\supp^T$ as in part~(ii) of Definition~\ref{def:coded-prae-dilator}. On the other hand, there is a small difference between Girard's pre-dilators and our prae-dilators (hence the particular spelling): The former fulfill a monotonicity condition that is automatic for well-orders (i.\,e.~in the case of dilators).

As promised above, we now show how the function $\alpha\mapsto\omega^\alpha$ from ordinal arithmetic can be implemented as a dilator:

\begin{example}\label{ex:omega-X}
For each order $X=(X,<_X)$ we consider the set
\begin{equation*}
\omega^X=\{\omega^{x_1}+\dots+\omega^{x_n}\,|\,x_1\geq_X\dots\geq_X x_n\}
\end{equation*}
with the lexicographic order (Cantor normal forms). Each embedding $f:X\rightarrow Y$ induces an embedding $\omega^f:\omega^X\rightarrow\omega^Y$, which is give by the clause
\begin{equation*}
\omega^f(\omega^{x_1}+\dots+\omega^{x_n})=\omega^{f(x_1)}+\dots+\omega^{f(x_n)}.
\end{equation*}
To define a family of functions $\supp^\omega_X:\omega^X\rightarrow[X]^{<\omega}$ we set
\begin{equation*}
\supp^\omega_X(\omega^{x_1}+\dots+\omega^{x_n})=\{x_1,\dots,x_n\}.
\end{equation*}
It is straightforward to verify that the given functions (or their restrictions to the category of natural numbers) form a prae-dilator, and indeed a dilator: If~$X$ has order-type~$\alpha$, then $\omega^X$ has order-type $\omega^\alpha$, in the usual sense of ordinal arithmetic. In the context of reverse mathematics, the statement that $X\mapsto\omega^X$ preserves well-foundedness is equivalent to arithmetical comprehension (due to J.-Y.~Girard~\cite{girard87}; cf.~also the computability-theoretic proof by J.~Hirst~\cite{hirst94}). Concerning Definition~\ref{def:coded-prae-dilator-reconstruct}, we point out that $\langle\{1,\omega\},\omega^1+\omega^0\rangle$ lies in $D^\omega(\omega+1)$ while $\langle\{1,5,\omega\},\omega^2+\omega^0\rangle$ does not (since $\supp^\omega_3(\omega^2+\omega^0)=\{0,2\}\neq 3$). The map $\langle a,\sigma\rangle\mapsto T(\iota_a\circ\en_a)(\sigma)$ sends both pairs to the element $\omega^\omega+\omega^1\in\omega^{\omega+1}$.
\end{example}

Each (coded) dilator $T$ induces a function from ordinals to ordinals, given by
\begin{equation*}
\alpha\mapsto\otp(D^T(\alpha)),
\end{equation*}
where $\otp(X)$ is the order type of (i.\,e.~the ordinal isomorphic to) the well-order~$X$. To see that this function does not need to be normal we consider the transformation
\begin{equation*}
X\mapsto T(X)=X\cup\{\top\}
\end{equation*}
that extends an order $X$ by a new maximal element $\top$. We get a dilator by setting
\begin{equation*}
 T(f)(\sigma)=\begin{cases}
               f(\sigma) & \text{if $\sigma\in X$},\\
               \top      & \text{if $\sigma=\top$},
              \end{cases}
 \qquad
 \supp^T_X(\sigma)=\begin{cases}
               \{\sigma\} & \text{if $\sigma\in X$},\\
               \emptyset  & \text{if $\sigma=\top$}.
              \end{cases}
\end{equation*}
The induced function $\alpha\mapsto\alpha+1$ fails to be continuous at limit ordinals and does not have any fixed points. Before we restore the focus on normal functions, let us mention that dilators that induce discontinuous functions are at least as interesting: In~\cite{freund-equivalence,freund-categorical,freund-computable} it has been shown that $\Pi^1_1$-comprehension is equivalent to the statement that every dilator $T$ admits a certain type of collapsing function $\vartheta:T(X)\rightarrow X$ for some well-order~$X$ (note that $\vartheta$ cannot be fully order preserving if $T(X)=X\cup\{\top\}$).

In order to analyse the example from the previous paragraph we observe that~$T$ does not preserve initial segments: It can happen that the range of $f:X\rightarrow Y$ is an initial segment of $Y$, while the range of $T(f):T(X)\rightarrow T(Y)$ fails to be an initial segment of $T(Y)$ (since it contains the element~$\top$). Indeed, Aczel~\cite{aczel-phd,aczel-normal-functors} and Girard~\cite{girard-pi2} have observed that preservation of initial segments ensures continuity at limit ordinals (cf.~the proof of Proposition~\ref{thm:normal-dil-fct} below). We will be particularly interested in initial segments of the form
\begin{equation*}
X\!\restriction\! x=\{x'\in X\,|\,x'<_X x\},
\end{equation*}
where $x$ is an element of the order~$X$. The functions $\mu^T_n$ in the following definition correspond to the restrictions $f\!\restriction\!\alpha:\alpha\rightarrow f(\alpha)$ of a normal function~$f$. In this sense we can view $\mu^T$ as an internal version of $T$ (see also Example~\ref{ex:omega-X-normal} below).

\begin{definition}[$\rca_0$]\label{def:normal-dil}
 A normal (prae-) dilator consists of a coded (prae-) dilator $T$ and a natural family of embeddings $\mu^T_n:n\rightarrow T(n)$ such that
 \begin{equation*}
  \sigma\in T(n)\!\restriction\!\mu^T_n(m)\quad\Leftrightarrow\quad\supp^T_n(\sigma)\subseteq m=\{0,\dots,m-1\}
 \end{equation*}
 holds for any numbers $m<n$ and an arbitrary element $\sigma\in T(n)$.
\end{definition}

Note that the single element $\mu^T_1(0)\in T(1)$ determines the entire family $\mu^T$, due to naturality: For $\iota:1\rightarrow n$ with $\iota(0)=m$ we have $\mu^T_n(m)=T(\iota)(\mu^T_1(0))$. The following result from~\cite{freund-rathjen_derivatives} will be needed to extend $\mu^T$ beyond the natural numbers.

\begin{lemma}[$\rca_0$]
 Assume that $T=(T,\mu^T)$ is a normal prae-dilator. Then
 \begin{equation*}
  \supp^T_n(\mu^T_n(m))=\{m\}
 \end{equation*}
 holds for arbitrary numbers $m<n$.
\end{lemma}
\begin{proof}
 Writing $\mu^T_n(m)=T(\iota)\circ\mu^T_1(0)$ as above, the naturality of $\supp^T$ yields
 \begin{equation*}
  \supp^T_n(\mu^T_n(m))=[\iota]^{<\omega}(\supp^T_1(\mu^T_1(0)))\subseteq\rng(\iota)=\{m\}.
 \end{equation*}
 Now it suffices to observe that $\supp^T_n(\mu^T_n(m))=\emptyset$ must fail: Otherwise we would have $\supp^T_n(\mu^T_n(m))\subseteq m$ and hence $\mu^T_n(m)\in T(n)\!\restriction\!\mu^T_n(m)$.
\end{proof}

In particular the lemma provides $\supp^T_1(\mu^T_1(0))=1$. Considering Definition~\ref{def:coded-prae-dilator-reconstruct}, this is needed to justify the following:

\begin{definition}[$\rca_0$]
 Given a normal prae-dilator $T$, we define a family of functions $D^{\mu^T}_X:X\rightarrow D^T(X)$ by setting
 \begin{equation*}
  D^{\mu^T}_X(x)=\langle\{x\},\mu^T_1(0)\rangle
 \end{equation*}
 for any order $X$ and any element $x\in X$.
\end{definition}

The following result tells us that the defining property of $\mu^T$ extends beyond the category of natural numbers. We refer to~\cite[Proposition~2.11]{freund-rathjen_derivatives} for a proof.

\begin{proposition}[$\rca_0$]\label{prop:normal-dilator-extend}
 Consider a normal prae-dilator~$T$. We have
 \begin{equation*}
  \langle a,\sigma\rangle\in D^T(X)\!\restriction\!D^{\mu^T}_X(x)\quad\Leftrightarrow\quad a\subseteq X\!\restriction\! x
 \end{equation*}
 for any linear order $X$ and any element $\langle a,\sigma\rangle\in D^T(X)$. Furthermore, the functions $D^{\mu^T}_X:X\rightarrow D^T(X)$ form a natural family of embeddings.
\end{proposition}

Let us extend Example~\ref{ex:omega-X} to accommodate the new notions:

\begin{example}\label{ex:omega-X-normal}
 To turn $X\mapsto\omega^X$ into a normal dilator we consider the functions
 \begin{equation*}
  \mu^\omega_X:X\rightarrow\omega^X,\qquad\mu^\omega_X(x)=\omega^x.
 \end{equation*}
Observe that $\omega^{x_1}+\dots+\omega^{x_n}\in\omega^X\!\restriction\!\omega^x$ is equivalent to $\{x_1,\dots,x_n\}\subseteq X\!\restriction\!x$, as required by Definition~\ref{def:normal-dil}. Let us also point out that the aforementioned function $\langle a,\sigma\rangle\mapsto T(\iota_a\circ\en_a)(\sigma)$ maps $D^{\mu^\omega}_X(x)=\langle\{x\},\omega^0\rangle\in D^\omega(X)$ to $\omega^x\in\omega^X$.
\end{example}

In a suitable meta theory one can establish the following result, which is due to Aczel~\cite[Theorem~2.11]{aczel-phd}. The given proof is very similar to the one in~\cite{freund-rathjen_derivatives}.

\begin{theorem}\label{thm:normal-dil-fct}
 Any normal dilator induces a normal function.
\end{theorem}
\begin{proof}
 Let $T$ be a normal dilator. Considering Definition~\ref{def:coded-prae-dilator-reconstruct}, we see that
 \begin{equation*}
  D^T(X\!\restriction\!x)=\{\langle a,\sigma\rangle\in D^T(X)\,|\,a\subseteq X\!\restriction\!x\}
 \end{equation*}
 is a suborder of $D^T(X)$ (see~\cite[Lemma~2.6]{freund-rathjen_derivatives} for a more general result). Since $T$ is normal, the previous proposition allows us to conclude
 \begin{equation*}
  D^T(X\!\restriction\!x)=D^T(X)\!\restriction\!D^{\mu^T}_X(x).
 \end{equation*}
 To show that the function $\alpha\mapsto\otp(D^T(\alpha))$ induced by $T$ is strictly increasing we consider $\alpha<\beta$. The usual set-theoretic definition of ordinals yields~$\beta\!\restriction\!\alpha=\alpha$. Using the above we obtain
 \begin{equation*}
  \otp(D^T(\alpha))=\otp(D^T(\beta\!\restriction\!\alpha))=\otp(D^T(\beta)\!\restriction\!D^{\mu^T}_\beta(\alpha))<\otp(D^T(\beta)),
 \end{equation*}
 as required. To show that $\alpha\mapsto\otp(D^T(\alpha))$ is continuous it remains to prove
 \begin{equation*}
  \otp(D^T(\lambda))\leq\textstyle\sup_{\alpha<\lambda}\otp(D^T(\alpha))
 \end{equation*}
 for an arbitrary limit ordinal~$\lambda$. Given any element $\langle a,\sigma\rangle\in D^T(\lambda)$, we pick an~$\alpha<\lambda$ with $a\subseteq\alpha=\lambda\!\restriction\!\alpha$. By the above we get $\langle a,\sigma\rangle\in D^T(\lambda)\!\restriction\!D^{\mu^T}_\lambda(\alpha)$ and hence
 \begin{equation*}
  \otp(D^T(\lambda)\!\restriction\!\langle a,\sigma\rangle)<\otp(D^T(\lambda)\!\restriction\!D^{\mu^T}_\lambda(\alpha))=\otp(D^T(\alpha))\leq\textstyle\sup_{\alpha<\lambda}\otp(D^T(\alpha)).
 \end{equation*}
 Since $\langle a,\sigma\rangle\in D^T(\lambda)$ was arbitrary this yields the desired inequality.
\end{proof}

Once normal prae-dilators have been defined, it is straightforward to find a suitable notion of fixed point (the reader may wish to compare this with the categorical definition of derivatives in~\cite{freund-rathjen_derivatives}, which is considerably more involved):

\begin{definition}[$\rca_0$]\label{def:fixed-point}
 A fixed point of a normal prae-dilator $T$ consists of an order~$X$ and an embedding $\xi:D^T(X)\rightarrow X$. We say that the fixed point is well-founded if the order~$X$ has this property.
\end{definition}

The reader might wonder whether the function $\xi$ from the previous definition should be an isomorphism. One could also focus on the initial fixed point of $T$, which should be embeddable into any other. In view of Theorems~\ref{thm:fixed-point-additional} and~\ref{thm:main-result} our main result remains valid for fixed points with these additional properties. The following observation, which requires a suitable base theory, provides an extensional justification for the given definition:

\begin{corollary}
 Consider a normal dilator~$T$. If $\alpha$ and $\xi:D^T(\alpha)\rightarrow\alpha$ form a fixed point of $T$, then we have $\alpha=\otp(D^T(\alpha))$.
\end{corollary}
\begin{proof}
 Proposition~\ref{prop:normal-dilator-extend} tells us that $D^{\mu^T}_\alpha:\alpha\rightarrow D^T(\alpha)$ is an embedding. Hence we have $\alpha\leq\otp(D^T(\alpha))$. Conversely, the embedding $\xi$ witnesses $\alpha\geq\otp(D^T(\alpha))$.
\end{proof}

\section{From fixed point to induction}\label{sect:fp-to-ind}

In this section we deduce $\Pi^1_1$-induction along the natural numbers from the assumption that every normal dilator has a well-founded fixed point. To achieve this goal we will give a precise version of the informal argument from the introduction of the present paper.

As in the informal argument, the Kleene normal form theorem implies that a given $\Pi^1_1$-statement $\varphi\equiv\varphi(n)$ corresponds to a family
\begin{equation*}
\mathcal T=\{(n,\sigma)\,|\,\sigma\in\mathcal T_n\}
\end{equation*}
of trees $\mathcal T_n\subseteq\mathbb N^{<\omega}$, such that $\mathcal T_n$ is well-founded if, and only if, the instance $\varphi(n)$ holds. Here $X^{<\omega}$ denotes the tree of finite sequences with entries in $X$, ordered by end extension. Any subtree $\mathcal T\subseteq X^{<\omega}$ will be called an $X$-tree. When we speak of a family of trees we will assume that it is indexed by the natural numbers, unless indicated otherwise. Hence the above expresses that $\Pi^1_1$-statements with a distinguished number variable correspond to families of $\mathbb N$-trees.

If $X=(X,<_X)$ is a linear order, then any $X$-tree $\mathcal T$ is totally ordered by the Kleene-Brouwer order (also known as Lusin-Sierpi\'nski order) with respect to $X$. The latter compares $\sigma^i=\langle\sigma^i_0,\dots,\sigma^i_{k_i-1}\rangle\in\mathcal T$ according to the clause
\begin{equation*}
\sigma^1<_{\operatorname{KB}(X)}\sigma^2\quad\Leftrightarrow\quad
\begin{cases}
\text{either $\sigma^1$ is a proper end extension of $\sigma^2$,}\\
\text{or we have $\sigma^1_j<_X\sigma^2_j$ and $\forall_{i<j}\,\sigma^1_i=\sigma^2_i$ for some $j$.}
\end{cases}
\end{equation*}
We will omit the reference to $X$ when $X=\mathbb N$ carries the usual order. Recall that a function $f:\mathbb N\rightarrow X$ is called a branch of an $X$-tree $\mathcal T$ if we have
\begin{equation*}
f[n]=\langle f(0),\dots,f(n-1)\rangle\in\mathcal T
\end{equation*}
for every number $n$. Given an $X$-tree $\mathcal T$ for a well-order~$X$, it is equivalent to assert that $\mathcal T$ is well-founded with respect to end extensions, that $\mathcal T$ has no branch, and that the Kleene-Brouwer order with respect to $X$ is well-founded on $\mathcal T$. It is well-known that the equivalence can be proved in $\aca_0$ (cf.~\cite[Lemma~V.1.3]{simpson09}).

Using the terminology that we have introduced, the premise of $\Pi^1_1$-induction along the natural numbers can be expressed in the following form:

\begin{definition}[$\aca_0$]\label{def:progressive}
Consider a family $\mathcal T$ of $\mathbb N$-trees. If we have
\begin{equation*}
\text{``$\mathcal T_n$ is well-founded"}\rightarrow\text{``$\mathcal T_{n+1}$ is well-founded"},
\end{equation*}
then we say that $\mathcal T$ is progressive at $n$. The family $\mathcal T$ is called progressive if it is progressive at every $n\in\mathbb N$ and $\mathcal T_0$ is well-founded.
\end{definition}

Recall the function $h$ from the informal argument given in the introduction. In order to represent this function we will construct a family of prae-dilators $H[n]$ such that $X\mapsto D^{H[n]}(X)$ preserves well-foundedness if, and only if, a given family~$\mathcal T$ of $\mathbb N$-trees is progressive at~$n$. Unfortunately the orders $D^{H[n]}(X)$ that arise from Definition~\ref{def:coded-prae-dilator-reconstruct} are somewhat hard to understand. For this reason we will first give an ad hoc definition of orders~$H[n](X)$. In a second step we will define (coded) prae-dilators $H[n]$ with $D^{H[n]}(X)\cong H[n](X)$. The following approach is inspired by D.~Normann's proof that the notion of dilator is $\Pi^1_2$-complete (see~\cite[Annex~8.E]{girard-book-part2}; cf.~also the similar argument in~\cite[Section~3]{freund-rathjen_derivatives}): Assuming that $H[n](X)$ is ill-founded for some well-order~$X$, we must ensure that $\mathcal T_n$ is well-founded while $\mathcal T_{n+1}$ is not. The idea is to construct $H[n](X)$ as a tree. Along each branch one searches for an embedding of $\mathcal T_n$ into $X$ and, simultaneously, for a branch in $\mathcal T_{n+1}$. In order to make this precise we need one additional construction: Let us define
\begin{equation*}
X^\top=X\cup\{\top\}
\end{equation*}
as the extension of a given order $X$ by a new maximal element $\top$. If we map each embedding $f:X\rightarrow Y$ to the embedding
\begin{equation*}
f^\top:X^\top\rightarrow Y^\top,\qquad f^\top(\sigma)=\begin{cases}
f(\sigma) & \text{if $\sigma\in X$,}\\
\top        & \text{if $\sigma=\top$,}
\end{cases}
\end{equation*}
then we obtain an endofunctor of linear orders (and indeed a dilator). The fact that $\top$ is maximal will be relevant in some constructions further below, but in the following definition it is not: We simply need a default value for functions into $X^\top$ (cf.~the choice of elements $x_i$ in the proof of Proposition~\ref{prop:reconstruct-h}).

\begin{definition}[$\aca_0$]\label{def:reconstruct-h}
Consider a family $\mathcal T$ of $\mathbb N$-trees and a natural number~$n$. For each order $X$ we define $H[n](X)=H[\mathcal T,n](X)$ as the tree of all sequences
\begin{equation*}
\langle\langle x_0,s_0\rangle,\dots,\langle x_{k-1},s_{k-1}\rangle\rangle\in(X^\top\times\mathbb N)^{<\omega}
\end{equation*}
that satisfy the following conditions:
\begin{enumerate}[label=(\roman*)]
\item For any $i,j<k$ that code elements $i\lkb j$ of $\mathcal T_n$, we have $x_i<_{X^\top}x_j$.
\item We have $\langle s_0,\dots,s_{k-1}\rangle\in\mathcal T_{n+1}$.
\end{enumerate}
The tree $H[n](X)$ carries the Kleene-Brouwer order with respect to $X^\top\times\mathbb N$ (where $\langle x,s\rangle$ preceeds $\langle x',s'\rangle$ if we have either $x<_{X^\top}x'$ or $x=x'$ and $s<s'$ in $\mathbb N$).
\end{definition}

Let us verify the crucial property:

\begin{proposition}[$\aca_0$]\label{prop:reconstruct-h}
A family $\mathcal T$ of $\mathbb N$-trees is progressive at $n\in\mathbb N$ if, and only if, the order $H[\mathcal T,n](X)$ is well-founded for every well-order~$X$.
\end{proposition}
\begin{proof}
To establish the contrapositive of the first direction we assume that $H[n](X)$ is ill-founded for some well-order~$X$. Since $X^\top\times\mathbb N$ is well-founded, the characteristic property of the Kleene-Brouwer order yields a branch $f:\mathbb N\rightarrow X^\top\times\mathbb N$ in the tree~$H[n](X)$. Writing $f(i)=\langle x_i,s_i\rangle$, it is straightforward to observe that
\begin{equation*}
\mathcal T_n\ni i\mapsto x_i\in X^\top
\end{equation*}
is an embedding into the well-order $X^\top$, while
\begin{equation*}
i\mapsto\langle s_0,\dots,s_{i-1}\rangle
\end{equation*}
is a branch in $T_{n+1}$. Hence $\mathcal T_n$ is well-founded while $\mathcal T_{n+1}$ is not, which means that $\mathcal T$ fails to be progressive at~$n$. Aiming at the contrapositive of the other direction, we assume that $\mathcal T_n$ is well-founded while $\mathcal T_{n+1}$ has a branch $i\mapsto s_i$. We set $X=\mathcal T_n$ (with the Kleene-Brouwer order) and define
\begin{equation*}
x_i=\begin{cases}
i      & \text{if $i\in X$},\\
\top & \text{otherwise}.
\end{cases}
\end{equation*}
It is straightforward to see that $i\mapsto\langle x_i,s_i\rangle$ is a branch in the tree $H[n](X)$, so that the latter is ill-founded, even though $X$ is a well-order.
\end{proof}

As explained above, the next task is to define a coded prae-dilator $H[n]$ such that $D^{H[n]}(X)\cong H[n](X)$ holds for any order~$X$. The values $H[n](m)$, which this prae-dilator assigns to the finite orders $m=\{0,\dots,m-1\}$, coincide with those from Definition~\ref{def:reconstruct-h}. It remains to define the action on morphisms, as well as the support functions. It will be useful to formulate the following definition for arbitrary orders, rather than just for finite orders represented by natural numbers.

\begin{definition}[$\aca_0$]
With each order embedding $f:X\rightarrow Y$ we associate a function $H[n](f):H[n](X)\rightarrow H[n](Y)$, defined by
\begin{equation*}
H[n](f)(\langle\langle x_0,s_0\rangle,\dots,\langle x_{k-1},s_{k-1}\rangle\rangle)=\langle\langle f^\top(x_0),s_0\rangle,\dots,\langle f^\top(x_{k-1}),s_{k-1}\rangle\rangle.
\end{equation*}
To define a family of functions $\supp^{H[n]}_X:H[n](X)\rightarrow[X]^{<\omega}$ we set
\begin{equation*}
\supp^{H[n]}_X(\langle\langle x_0,s_0\rangle,\dots,\langle x_{k-1},s_{k-1}\rangle\rangle)=\{x_i\,|\,i<k\text{ and }x_i\neq\top\}
\end{equation*}
for each order $X$.
\end{definition}

It is straightforward to check that the conditions from Definition~\ref{def:coded-prae-dilator} are satisfied:

\begin{lemma}[$\aca_0$]\label{lem:H-n-prae-dil}
By restricting the previous definitions to the category of natural numbers we obtain a coded prae-dilator $H[n]=H[\mathcal T,n]$, for any family $\mathcal T$ of $\mathbb N$-trees and any number~$n$.
\end{lemma}

The following justifies the ad hoc definition of the orders~$H[n](X)$.

\begin{lemma}[$\aca_0$]\label{lem:H-orders-dilator}
We have $D^{H[n]}(X)\cong H[n](X)$ for any order~$X$.
\end{lemma}
\begin{proof}
The claim is a special case of a general result about the connection between coded and class-sized prae-dilators, established in~\cite[Proposition~2.5]{freund-computable} (cf.~also Section~\ref{sect:representation} of the present paper). The desired isomorphism is given by
\begin{equation*}
D^{H[n]}(X)\ni\langle a,\sigma\rangle\mapsto H[n](\iota_a\circ\en_a)(\sigma)\in H[n](X),
\end{equation*}
where $\en_a:|a|\rightarrow a$ is the enumeration of $a$ and $\iota_a:a\hookrightarrow X$ is the inclusion. To show that the given map is order preserving (and hence injective) one argues exactly as in the proof of the general case (see~\cite{freund-computable} for details). It is instructive to look at the proof that the given map is surjective: Given an arbitrary element
\begin{equation*}
\tau=\langle\langle x_0,s_0\rangle,\dots,\langle x_{k-1},s_{k-1}\rangle\rangle\in H[n](X),
\end{equation*}
we set $a=\supp^{H[n]}_X(\tau)$. Since $\iota_a\circ\en_a:|a|\rightarrow X$ has range $a$ we can define
\begin{equation*}
m_i=\begin{cases}
\text{``the unique $m<|a|$ with $\iota_a\circ\en_a(m)=x_i$"} & \text{if $x_i\neq\top$},\\
\top & \text{otherwise}.
\end{cases}
\end{equation*}
It follows that we have $(\iota_a\circ\en_a)^\top(m_i)=x_i$ for all $i<k$. Since $(\iota_a\circ\en_a)^\top$ is order preserving we can conclude that
\begin{equation*}
\sigma=\langle\langle m_0,s_0\rangle,\dots,\langle m_{k-1},s_{k-1}\rangle\rangle
\end{equation*}
is an element of~$H[n](|a|)$. In view of
\begin{equation*}
[\iota_a\circ\en_a]^{<\omega}(\supp^{H[n]}_{|a|}(\sigma))=\{\iota_a\circ\en_a(m_i)\,|\,i<k\text{ and }m_i\neq\top\}=a
\end{equation*}
we have $\supp^{H[n]}_{|a|}(\sigma)=|a|$, which yields $\langle a,\sigma\rangle\in D^{H[n]}(X)$. By construction we have $\tau=H[n](\iota_a\circ\en_a)(\sigma)$, as required for surjectivity.
\end{proof}

In view of Definition~\ref{def:coded-dilator}, the previous considerations yield the following result, which completes the reconstruction of the function $h$ that appears in the informal argument from the introduction:

\begin{corollary}[$\aca_0$]
A family $\mathcal T$ of $\mathbb N$-trees is progressive at $n$ if, and only if, the prae-dilator $H[\mathcal T,n]$ is a dilator.
\end{corollary}

To proceed we recall the functions $h_0(\gamma)=\sup_{n\in\mathbb N}h(n,\gamma)$ and
\begin{equation*}
f(\delta)=\otp(\mathcal T_0)+1+\sum_{\gamma<\delta}(h_0(\gamma)+1)
\end{equation*}
from the informal argument. Supremum and infinite sum can be implemented as dependent sums: Given an order $X$ and an $X$-indexed family of orders $Y_x$, the set
\begin{equation*}
\Sigma_{x\in X}Y_x=\{\langle x,y\rangle\,|\,x\in X\text{ and }y\in Y_x\}
\end{equation*}
is ordered according to the clause
\begin{equation*}
\langle x,y\rangle<_{\Sigma_{x\in X}Y_x}\langle x',y'\rangle\quad\Leftrightarrow\quad\begin{cases}
\text{either $x<_X x'$,}\\
\text{or $x=x'$ and $y<_{Y_x}y'$.}
\end{cases}
\end{equation*}
The elements of a binary sum $Y_0+Y_1=\Sigma_{i\in\{0,1\}}Y_i$ will be written as $\langle\bot,y_0\rangle$ and~$y_1$ rather than $\langle 0,y_0\rangle$ resp.~$\langle 1,y_1\rangle$ (intuitively, this means that we read the definition of $f$ as a single sum over $1+\delta$). If $F=(F,\mu^F)$ is a normal dilator that represents $f$, then the values of $\mu^F$ should correspond to the smallest elements of the summands~$h_0(\gamma)+1$. Since there does not appear to be a uniform way to choose these elements, we slightly deviate from the definition of $f$ and consider the summands $1+h_0(\gamma)+1$ instead. It will be convenient to represent the latter by a single dependent sum. For this purpose we consider the order
\begin{equation*}
 \npm=\{-1\}\cup\mathbb N\cup\{\infty\},
\end{equation*}
which extends the natural numbers by a new minimal and maximal element. We now define $H[-1]=H[\infty]$ as the constant dilator with values
\begin{equation*}
 H[-1](X)=H[\infty](X)=\{\star\},
\end{equation*}
for a new symbol $\star$ (note $H[-1](f)(\star)=\star$ and $\supp^{H[-1]}_X(\star)=\emptyset$). Assuming $\delta\cong X$ and $\gamma\cong X\!\restriction\!x$, the summand $1+h_0(\gamma)+1$ can then be represented (or rather bounded) by the dependent sum
\begin{equation*}
 \Sigma_{n\in\npm} H[n](X\!\restriction\!x),
\end{equation*}
which extends the sum $\Sigma_{n\in\mathbb N} H[n](X\!\restriction\!x)$ by a minimal element $\langle -1,\star\rangle$ and a maximal element $\langle\infty,\star\rangle$. Let us now define orders $F(X)$ that represent the values $f(\delta)$ of the function from the informal argument given in the introduction. We will later equip $F$ with the structure of a coded normal prae-dilator such that $D^F(X)\cong F(X)$ holds for any order~$X$.

\begin{definition}[$\aca_0$]
 For any family $\mathcal T$ of $\mathbb N$-trees and any order~$X$ we define
 \begin{equation*}
  F(X)=F[\mathcal T](X)=\mathcal T_0^\top+\Sigma_{x\in X}\Sigma_{n\in\npm} H[\mathcal T,n](X\!\restriction\!x),
 \end{equation*}
 with the usual order on a dependent sum.
\end{definition}

According to the above explanations, elements of $F(X)$ have the form $\langle\bot,\sigma\rangle$ with $\sigma\in\mathcal T_0\cup\{\top\}$ or $\langle x,\langle n,\sigma\rangle\rangle$ with $x\in X$, $n\in\npm$ and $\sigma\in H[n](X\!\restriction\!x)$. Elements of the second form will be written as $\langle x,n,\sigma\rangle$, with one pair of angle brackets omitted. At the beginning of this section we have expressed $\Pi^1_1$-induction along the natural numbers in terms of a family $\mathcal T$ of $\mathbb N$-trees. In this setting, the conclusion of induction amounts to the statement that all trees $\mathcal T_n$ are well-founded. The following result implies that this is the case if $F$ has a well-founded fixed point.

\begin{theorem}[$\aca_0$]\label{thm:fixed-point-embedding}
 Consider a family $\mathcal T$ of $\mathbb N$-trees. Given an order $X$ with an embedding $\xi:F[\mathcal T](X)\rightarrow X$, we can construct an embedding $J:\Sigma_{n\in\mathbb N}\mathcal T_n^\top\rightarrow X$.
\end{theorem}
\begin{proof}
 The informal argument from the introduction would suggest to construct the branches $\mathcal T_n^\top\ni\sigma\mapsto J(\langle n,\sigma\rangle)$ by recursion on~$n$, but the required recursion principle is not available in our base theory. Remarkably, the reconstruction of the informal argument in terms of dilators is sufficiently finitistic to allow for a definition of $J$ by course-of-values recursion over the codes of pairs in $\Sigma_{n\in\mathbb N}\mathcal T_n^\top$. To ensure that the required values of $J$ are available in the recursion step we make two assumptions about the coding of pairs and sequences: Firstly, we assume that
 \begin{equation*}
  \langle n,\top\rangle<_{\mathbb N}\langle n+1,\sigma\rangle
 \end{equation*}
 holds for any $\sigma\in\mathcal T_{n+1}^\top$ (we write $<_{\mathbb N}$ to stress that the codes are compared with respect to the usual order on the natural numbers). If we agree to represent the symbol $\top$ by the number zero, then this is satisfied for the usual Cantor coding of pairs. Secondly, we assume that the code of a finite sequence bounds its length, so that we have
 \begin{equation*}
  \langle n,i\rangle<_{\mathbb N}\langle n+1,\langle s_0,\dots,s_{k-1}\rangle\rangle
 \end{equation*}
 for any element $\langle s_0,\dots,s_{k-1}\rangle\in\mathcal T_{n+1}$ and all $i<k$. The values
 \begin{equation*}
  J(\langle 0,\sigma\rangle)=\xi(\langle\bot,\sigma\rangle)
 \end{equation*}
 are defined without recursive calls. To specify the remaining values we abbreviate
 \begin{equation*}
  J_n(i)=\begin{cases}
         J(\langle n,i\rangle) & \text{if $i$ codes an element of $\mathcal T_n$,}\\
         \top & \text{otherwise.}
        \end{cases}
 \end{equation*}
 We can now complete the recursive definition of $J$ by setting
 \begin{align*}
  J(\langle n+1,\top\rangle)&=\xi(\langle J(\langle n,\top\rangle),\infty,\star\rangle),\\
  J(\langle n+1,\langle s_0,\dots,s_{k-1}\rangle\rangle)&=\xi(\langle J(\langle n,\top\rangle),n,\langle\langle J_n(0),s_0\rangle,\dots,\langle J_n(k-1),s_{k-1}\rangle\rangle\rangle).
 \end{align*}
 To show that this defines an embedding $J:\Sigma_{n\in\mathbb N}\mathcal T_n^\top\rightarrow X$ we verify the following properties by simultaneous induction on~$j$:
 \begin{enumerate}[label=(\roman*)]
  \item If $j$ codes an element of $\Sigma_{n\in\mathbb N}\mathcal T_n^\top$, then we have $J(j)\in X$.
  \item If $j_0,j_1<j$ code elements $j_0<_{\Sigma_{n\in\mathbb N}\mathcal T_n^\top}j_1$, then we have $J(j_0)<_XJ(j_1)$.
  \item If we have $j=\langle n,\sigma\rangle$ for some $\sigma\in\mathcal T_n$, then we get $J(j)<_XJ(\langle n,\top\rangle)$.
 \end{enumerate}
 Claim~(i) is most interesting for $j=\langle n+1,\langle s_0,\dots,s_{k-1}\rangle\rangle$, where the second component of the pair lies in $\mathcal T_{n+1}$. To show that $J(j)$ lies in $X$ it suffices to establish
 \begin{equation*}
  \langle\langle J_n(0),s_0\rangle,\dots,\langle J_n(k-1),s_{k-1}\rangle\rangle\in H[n](X\!\restriction\!J(\langle n,\top\rangle)).
 \end{equation*}
 In view of Definition~\ref{def:reconstruct-h} this requires $J_n(i)\in(X\!\restriction\!J(\langle n,\top\rangle))^\top$ for all $i<k$, which can be deduced from parts~(i) and~(iii) of the simultaneous induction hypothesis. We also need $J_n(i)<_X J_n(i')$ for all $i,i'<k$ that code elements $i\lkb i'$ of $\mathcal T_n$. This follows from the induction hypothesis for~(ii). To verify the induction step for~(ii) one needs to distinguish several cases. We simplify the notation by writing
 \begin{equation*}
  J(\langle n+1,\sigma\rangle)=\xi(\langle J(\langle n,\top\rangle),J^n(\sigma)\rangle)
 \end{equation*}
 for both $\sigma\in\mathcal T_{n+1}$ and $\sigma=\top$ (to be read as an implicit definition of $J^n(\sigma)$). In~the case of an inequality
 \begin{equation*}
  \langle 0,\sigma\rangle<_{\Sigma_{n\in\mathbb N}\mathcal T_n^\top}\langle n+1,\sigma'\rangle
 \end{equation*}
 we observe that $\langle\bot,\sigma\rangle$ and $\langle J(\langle n,\top\rangle),J^n(\sigma')\rangle$ lie in the left resp.~right summand of~$F(X)$. Since $\xi$ is order preserving we can infer
 \begin{equation*}
  J(\langle 0,\sigma\rangle)=\xi(\langle\bot,\sigma\rangle)<_X\xi(\langle J(\langle n,\top\rangle),J^n(\sigma')\rangle)=J(\langle n+1,\sigma'\rangle).
 \end{equation*}
 Let us now consider an inequality
 \begin{equation*}
  \langle n+1,\sigma\rangle<_{\Sigma_{n\in\mathbb N}\mathcal T_n^\top}\langle n'+1,\sigma'\rangle
 \end{equation*}
 with $n<n'$. The induction hypothesis yields $J(\langle n,\top\rangle)<_X J(\langle n',\top\rangle)$ and hence
 \begin{equation*}
  \langle J(\langle n,\top\rangle),J^n(\sigma)\rangle<_{F(X)}\langle J(\langle n',\top\rangle),J^{n'}(\sigma')\rangle.
 \end{equation*}
 To conclude we apply $\xi$ to both sides. Finally, we look at an inequality
 \begin{equation*}
  \langle n+1,\sigma\rangle<_{\Sigma_{n\in\mathbb N}\mathcal T_n^\top}\langle n+1,\sigma'\rangle
 \end{equation*}
 with $\sigma<_{\mathcal T_{n+1}^\top}\sigma'$. It is straightforward to see that $J^n(\sigma)$ preceeds $J^n(\sigma')$ in the order $\Sigma_{n\in\mathbb N^\infty_{-1}}H[n](X\!\restriction\!J(\langle n,\top\rangle))$, both for $\sigma'\in\mathcal T_{n+1}$ and for $\sigma'=\top$. We thus get
 \begin{equation*}
  J(\langle n+1,\sigma\rangle)=\xi(\langle J(\langle n,\top\rangle),J^n(\sigma)\rangle)<_X\xi(\langle J(\langle n,\top\rangle),J^n(\sigma')\rangle)=J(\langle n+1,\sigma'\rangle),
 \end{equation*}
 as desired. Claim~(iii) is straightforward once we know that the relevant values of $J$ lie in~$X$ (in case $n=m+1$ one uses the inequality $m<\infty$ in $\npm$).
\end{proof}

Our next goal is to show that $F=F[\mathcal T]$ can be extended into a (coded) normal prae-dilator such that we have $D^F(X)\cong F(X)$ for any order~$X$. To explain the definition  of $\mu^F$ we recall that $\langle -1,\star\rangle$ is the minimal element of $\Sigma_{n\in\npm}H[n](X)$, independently of the order~$X$.

\begin{definition}[$\aca_0$]
Given an order embedding $f:X\rightarrow Y$, we define a function $F(f):F(X)\rightarrow F(Y)$ by
\begin{align*}
F(f)(\langle\bot,\sigma\rangle)&=\langle\bot,\sigma\rangle,\\
F(f)(\langle x,n,\sigma\rangle)&=\langle f(x),n,H[n](f\!\restriction\!x)(\sigma)\rangle,
\end{align*}
where $f\!\restriction\!x:X\!\restriction\!x\rightarrow Y\!\restriction\!f(x)$ is the restriction of~$f$. In order to define a family of functions $\supp^F_X:F(X)\rightarrow[X]^{<\omega}$ we stipulate
\begin{align*}
\supp^F_X(\langle\bot,\sigma\rangle)&=\emptyset,\\
\supp^F_X(\langle x,n,\sigma\rangle)&=\{x\}\cup\supp^{H[n]}_{X\restriction x}(\sigma).
\end{align*}
Finally, we define functions $\mu^F_X:X\rightarrow F(X)$ by setting
\begin{equation*}
\mu^F_X(x)=\langle x,-1,\star\rangle
\end{equation*}
for each order~$X$ and each element $x\in X$.
\end{definition}

In order to apply Theorem~\ref{thm:fixed-point-embedding} we will invoke the principle that every normal dilator has a well-founded fixed point. For this purpose we need the following result:

\begin{proposition}[$\aca_0$]\label{prop:F-prae-dilator}
The restriction of the previous constructions to the category of natural numbers defines a coded normal prae-dilator $F=F[\mathcal T]$, for each family~$\mathcal T$ of $\mathbb N$-trees.
\end{proposition}
\begin{proof}
It is straightforward to check that $F$ is a functor and that $\supp^F$ is a natural transformation, invoking Lemma~\ref{lem:H-n-prae-dil} for the corresponding properties of $H[n]$. To verify the support condition from part~(ii) of Definition~\ref{def:coded-prae-dilator} we consider an element
\begin{equation*}
\tau=\langle m_0,n,\sigma\rangle\in F(m)
\end{equation*}
with $m_0<m$ and $\sigma\in H[n](m_0)$ (observe $m\!\restriction\!m_0=m_0$). In view of
\begin{equation*}
\supp^F_m(\tau)=\{m_0\}\cup\supp^{H[n]}_{m_0}(\sigma)
\end{equation*}
we write $|\supp^F_m(\tau)|=k+1$. Let $\iota_\tau\circ\en_\tau:k+1\rightarrow m$ denote the embedding with range $\supp^F_m(\tau)$, as in Definition~\ref{def:coded-prae-dilator}. Due to $\supp^{H[n]}_{m_0}(\sigma)\in[m_0]^{<\omega}$ we see that
\begin{equation*}
(\iota_\tau\circ\en_\tau)\!\restriction\!k:k\rightarrow\iota_\tau\circ\en_\tau(k)=m_0
\end{equation*}
has range $\supp^{H[n]}_{m_0}(\sigma)$. Hence the support condition for $H[n]$ yields
\begin{equation*}
 \sigma=H[n]((\iota_\tau\circ\en_\tau)\!\restriction\!k)(\sigma_0)
\end{equation*}
for some element $\sigma_0\in H[n](k)$ (in case $n\in\{-1,\infty\}$ we have $\sigma=\star=\sigma_0$). Setting
\begin{equation*}
 \tau_0=\langle k,n,\sigma_0\rangle\in F(k+1)
\end{equation*}
we get $\tau=F(\iota_\tau\circ\en_\tau)(\tau_0)$, as required by the support condition for~$F$. To prove that $F=(F,\mu^F)$ is normal we must establish
\begin{equation*}
 \tau<_{F(m)}\mu^F_m(k)\quad\Leftrightarrow\quad\supp^F_m(\tau)\subseteq\{0,\dots,k-1\},
\end{equation*}
for arbitrary numbers $k<m$ and any element $\tau\in F(m)$. Let us first assume that we are concerned with an element of the form $\tau=\langle\bot,\sigma\rangle$. In this case the left side of the equivalence is satisfied, since $\tau$ lies in the first summand of $F(m)$ while $\mu^F_m(k)=(k,-1,\star)$ lies in the second. The right side of the equivalence holds because of $\supp^F_m(\tau)=\emptyset$. Now consider an element
\begin{equation*}
 \tau=\langle m_0,n,\sigma\rangle.
\end{equation*}
Since $\langle -1,\star\rangle$ is the smallest element of $\Sigma_{n\in\npm}H[n](k)$ we have
\begin{equation*}
 \tau<_{F(m)}\mu^F_m(k)\quad\Leftrightarrow\quad m_0<k.
\end{equation*}
In view of $\supp^{H[n]}_{m_0}(\sigma)\subseteq\{0,\dots,m_0-1\}$ we also have
\begin{equation*}
 \supp^F_m(\tau)\subseteq\{0,\dots,k-1\}\quad\Leftrightarrow\quad m_0<k,
\end{equation*}
which completes the proof of the required equivalence.
\end{proof}

Let us also connect the orders $F(X)$ to the coded prae-dilator $F$:

\begin{lemma}[$\aca_0$]\label{lem:coded-class-F}
 We have $F(X)\cong D^{F}(X)$ for any order~$X$.
\end{lemma}
\begin{proof}
Yet again, this is an instance of the general result from~\cite[Proposition~2.5]{freund-computable}. As an alternative to the general result, the claim can be deduced from Lemma~\ref{lem:H-orders-dilator}: In view of that result (which readily extends to $n\in\{-1,\infty\}$) it suffices to show
\begin{equation*}
 \mathcal T^\top_0+\Sigma_{x\in X}\Sigma_{n\in\npm}D^{H[n]}(X\!\restriction\!x)\cong D^F(X).
\end{equation*}
Every element of the second summand on the left side has the form $\langle x,n,\langle a,\sigma\rangle\rangle$, for a finite subset $a\subseteq X\!\restriction\!x$ and an element $\sigma\in H[n](|a|)$ with $\supp^{H[n]}_{|a|}(\sigma)=|a|$. The desired isomorphism can now be specified by stipulating
\begin{align*}
 \langle\bot,\sigma\rangle&\mapsto\langle\emptyset,\langle\bot,\sigma\rangle\rangle,\\
 \langle x,n,\langle a,\sigma\rangle\rangle&\mapsto\langle\{x\}\cup a,\langle |a|,n,\sigma\rangle\rangle.
\end{align*}
To see that the values lie in $D^F(X)$ one observes $\langle |a|,n,\sigma\rangle\in F(|\{x\}\cup a|)$ and
\begin{equation*}
 \supp^F_{|\{x\}\cup a|}(\langle |a|,n,\sigma\rangle)=\{|a|\}\cup\supp^{H[n]}_{|a|}(\sigma)=|a|+1=|\{x\}\cup a|.
\end{equation*}
The fact that the given map is order preserving (and hence injective) is verified as in the proof of~\cite[Lemma~3.10]{freund-rathjen_derivatives}. To establish surjectivity we consider an element
\begin{equation*}
 \langle b,\langle m,n,\sigma\rangle\rangle\in D^F(X).
\end{equation*}
In view of Definition~\ref{def:coded-prae-dilator-reconstruct} we have $\langle m,n,\sigma\rangle\in F(|b|)$, which yields $\sigma\in H[n](m)$, and
\begin{equation*}
 |b|=\supp^F_{|b|}(\langle m,n,\sigma\rangle)=\{m\}\cup\supp^{H[n]}_m(\sigma).
\end{equation*}
Let $x$ be the largest element of $b$ and set $a=b\backslash\{x\}$. It is straightforward to conclude $m=|a|$ and $\supp^{H[n]}_m(\sigma)=\{0,\dots,|a|-1\}$, which yields $\langle a,\sigma\rangle\in D^{H[n]}(X\!\restriction\!x)$. Hence $\langle b,\langle m,n,\sigma\rangle\rangle$ arises as the image of $\langle x,n,\langle a,\sigma\rangle\rangle$, as needed for surjectivity.
\end{proof}

We can draw the following conclusion:

\begin{proposition}[$\aca_0$]\label{prop:F-prog-dilator}
 A family $\mathcal T$ of $\mathbb N$-trees is progressive if, and only if, the normal prae-dilator $F[\mathcal T]$ is a normal dilator.
\end{proposition}
\begin{proof}
 For the first direction we assume that $\mathcal T$ is progressive. According to Definition~\ref{def:progressive} this means that $\mathcal T_0$ is well-founded and that $\mathcal T_n$ is progressive at every~$n\in\mathbb N$. The latter implies that the maps $X\mapsto H[n](X)$ preserve well-foundedness, due to Proposition~\ref{prop:reconstruct-h}. We must show that $D^F(X)\cong F(X)$ is well-founded for any given well-order~$X$. Aiming at a contradiction, assume there is a descending sequence in
 \begin{equation*}
  F(X)=\mathcal T_0^\top+\Sigma_{x\in X}\Sigma_{n\in\npm} H[n](X\!\restriction\!x).
 \end{equation*}
 As $\mathcal T_0^\top$ is well-founded this sequence must stay within the second summand, so that we can write it as $k\mapsto\langle x_k,n_k,\sigma_k\rangle$. Since $X$ and $\npm$ are both well-founded, there must be values $x\in X$ and $n\in\npm$ such that we have $x_k=x$ and $n_k=n$ for all indices $k$ above some bound $K\in\mathbb N$. It follows that $K\leq k\mapsto\sigma_k$ is a descending sequence in $H[n](X\!\restriction\!x)$, contradicting the well-foundedness of that order (note that $H[-1](X\!\restriction\!x)=\{\star\}=H[\infty](X\!\restriction\!x)$ is well-founded in any case). For the other direction we assume that $F(X)\cong D^F(X)$ is well-founded for any well-order~$X$. We immediately learn that $\mathcal T_0$, which can be embedded into $F(\emptyset)$, is well-founded. In view of Proposition~\ref{prop:reconstruct-h} it remains to show that $X\mapsto H[n](X)$ preserves well-foundedness for any number $n$. Given a well-order~$X$, we observe that $X^\top=X\cup\{\top\}$ is a well-order that contains $X=X^\top\!\restriction\!\top$. The embedding
 \begin{equation*}
  H[n](X)\ni\sigma\mapsto\langle\top,n,\sigma\rangle\in F(X^\top)
 \end{equation*}
 witnesses the well-foundedness of $H[n](X)$.
\end{proof}

Putting things together, we can deduce the first direction of our main result:

\begin{theorem}[$\aca_0$]\label{thm:fixed-to-ind}
Assume that every normal dilator has a well-founded fixed point. Then (each instance of) $\Pi^1_1$-induction along the natural numbers holds.
\end{theorem}
\begin{proof}
 Given a $\Pi^1_1$-formula $\varphi(n)$ (possibly with parameters), we can use the Kleene normal form theorem (see~\cite[Lemma~V.1.4]{simpson09}) to find a $\Delta^0_0$-formula $\theta(\sigma,n)$ with
 \begin{equation*}
  \varphi(n)\leftrightarrow\forall_f\exists_m\theta(f[m],n),
 \end{equation*}
 where the universal quantifier ranges over functions $f:\mathbb N\rightarrow\mathbb N$. Now define a family $\mathcal T$ of $\mathbb N$-trees by stipulating
 \begin{equation*}
  \langle s_0,\dots,s_{k-1}\rangle\in\mathcal T_n\,\leftrightarrow\,\text{``we have $\neg\theta(\langle s_0,\dots,s_{i-1}\rangle,n)$ for all $i\leq k$''}.
 \end{equation*}
 Hence $\neg\varphi(n)$ is equivalent to the statement that $\mathcal T_n$ has a branch. If we equip $\mathcal T_n$ with the Kleene-Brouwer order, then we obtain
 \begin{equation*}
  \varphi(n)\leftrightarrow\text{``$\mathcal T_n$ is well-founded''}.
 \end{equation*}
 According to Proposition~\ref{prop:F-prae-dilator} the family $\mathcal T$ gives rise to a normal prae-dilator~$F[\mathcal T]$. To establish the induction principle we assume $\varphi(0)$ and $\forall_n(\varphi(n)\rightarrow\varphi(n+1))$. In view of Definition~\ref{def:progressive} these assumptions mean that $\mathcal T$ is progressive. We can then use Proposition~\ref{prop:F-prog-dilator} to infer that $\mathcal F[\mathcal T]$ is a normal dilator. Invoking the assumption of the present theorem we obtain a well-order~$X$ and an embedding
 \begin{equation*}
  \xi:F[\mathcal T](X)\cong D^{F[\mathcal T]}(X)\rightarrow X,
 \end{equation*}
 where the isomorphism comes from Lemma~\ref{lem:coded-class-F}. Now Theorem~\ref{thm:fixed-point-embedding} tells us that the dependent sum $\Sigma_{n\in\mathbb N}\mathcal T_n^\top$ can be embedded into~$X$. Since the latter is a well-order this ensures that all trees $\mathcal T_n$ are well-founded. We thus obtain $\forall_n\varphi(n)$, which is the conclusion of the desired induction principle.
\end{proof}

\section{From induction to well-founded fixed point}\label{sect:ind-to-wf}

In the first part of this section we describe a relativized notation system for the initial fixed point $\fix(T)$ of a given normal prae-dilator~$T$, working in $\rca_0$. Theorem~\ref{thm:fixed-to-ind} implies that $\rca_0$ cannot prove the principle that $\fix(T)$ is well-founded whenever~$T$ is a dilator. In the second part of the present section we show that this principle follows from $\Pi^1_1$-induction along the natural numbers. This reversal of Theorem~\ref{thm:fixed-to-ind} completes the proof of our main result.

Section~4 of~\cite{freund-rathjen_derivatives} contains a construction of the derivative $X\mapsto\partial T_X$ of a given normal prae-dilator~$T$. Our order $\fix(T)$ will coincide with the order $\partial T_0$ that arises from this construction. Since the definition of the full derivative $\partial T$ is considerably more involved than the definition of the single fixed point~$\fix(T)$, we think that it is nevertheless worthwhile to give an independent construction of the latter.

To motivate our construction we assume that we already have an order $\fix(T)$ that admits an embedding $\xi:D^T(\fix(T))\rightarrow\fix(T)$. According to Definition~\ref{def:coded-prae-dilator-reconstruct} the set $D^T(\fix(T))$ consists of pairs $\langle a,\sigma\rangle$ of a finite set $a\subseteq\fix(T)$ and an element $\sigma\in T(|a|)$ with $\supp^T_{|a|}(\sigma)=|a|$. The idea is that the value $\xi(\langle a,\sigma\rangle)\in\fix(T)$ can be represented by a term $\xi\langle a,\sigma\rangle$. This leads to the following:

\begin{definition}[$\rca_0$]\label{def:fix-T}
For each normal prae-dilator $T$ we define a set $\fix(T)$ of terms by the following inductive clause: Given a finite set $a\subseteq\fix(T)$ and an element $\sigma\in T(|a|)$ with $\supp^T_{|a|}(\sigma)=|a|$, we add a term $\xi\langle a,\sigma\rangle\in\fix(T)$.
\end{definition}

Note that $\fix(T)\neq\emptyset$ is equivalent to $T(0)\neq\emptyset$. To define the order relation on $\fix(T)$ we need a suitable length function $L_T:\fix(T)\rightarrow\mathbb N$. In the context of $\rca_0$ it will be important that quantifiers of the form $\forall_{s\in\fix(T)}(L_T(s)\leq n\rightarrow\dots)$ are bounded. For this purpose we ensure that $L_T(s)$ bounds the G\"odel number $\ulcorner s\urcorner$ of the term $s$ (we have $\ulcorner s\urcorner=s$ if the previous definition is already arithmetized). Inductively we set
\begin{equation*}
L_T(\xi\langle a,\sigma\rangle)=\max\{\ulcorner\xi\langle a,\sigma\rangle\urcorner,1+\textstyle\sum_{s\in a}2\cdot L_T(s)\}.
\end{equation*}
To define a relation $<_{\fix(T)}$ on $\fix(T)$ we will decide $\xi\langle a,\sigma\rangle<_{\fix(T)}\xi\langle b,\tau\rangle$ by recursion on $L_T(\xi\langle a,\sigma\rangle)+L_T(\xi\langle b,\tau\rangle)$. In the recursion step we may assume that the restriction of $<_{\fix(T)}$ to $a\cup b$ is already determined (note that $2\cdot L_T(s)<L_T(\xi\langle a,\sigma\rangle)$ for $s\in a$ allows us to decide $s<_{\fix(T)}s$). If this restriction is linear, then we may consider the unique function $|\iota_a^{a\cup b}|:|a|\rightarrow|a\cup b|$ with
\begin{equation*}
\en_{a\cup b}\circ |\iota_a^{a\cup b}|=\iota_a^{a\cup b}\circ\en_a,
\end{equation*}
where $\en_{a\cup b}:|a\cup b|\rightarrow a\cup b$ and $\en_a:|a|\rightarrow a$ are the unique increasing enumerations with respect to $<_{\fix(T)}$ and $\iota_a^{a\cup b}:a\hookrightarrow a\cup b$ is the inclusion map (similarly for~\mbox{$\iota_b^{a\cup b}:b\hookrightarrow a\cup b$}). The following is reminiscent of Definition~\ref{def:coded-prae-dilator-reconstruct}, which makes sense because we are aiming at an isomorphism between $\fix(T)$ and $D^T(\fix(T))$.

\begin{definition}[$\rca_0$]\label{def:fix-T-order}
To define a binary relation $<_{\fix(T)}$ on $\fix(T)$ we recursively stipulate that $\xi\langle a,\sigma\rangle<_{\fix(T)}\xi\langle\tau,b\rangle$ holds if, and only if, the restriction of $<_{\fix(T)}$ to $a\cup b$ is linear and we have $T(|\iota_a^{a\cup b}|)(\sigma)<_{T(|a\cup b|)}T(|\iota_b^{a\cup b}|)(\tau)$.
\end{definition}

As suggested by the notation, we have the following property:

\begin{lemma}[$\rca_0$]\label{lem:fix-T-linear}
The relation $<_{\fix(T)}$ is a linear order on $\fix(T)$, for any normal prae-dilator~$T$.
\end{lemma}
\begin{proof}
It is straightforward to see that $<_{\fix(T)}$ is irreflexive, invoking the same property of the orders~$<_{T(m)}$. To conclude one simultaneously verifies
\begin{gather*}
s<_{\fix(T)} t\lor s=t\lor t<_{\fix(T)}s,\\
s<_{\fix(T)} t\land t<_{\fix(T)} r\rightarrow s<_{\fix(T)} r,
\end{gather*}
by induction on $L_T(s)+L_T(t)$ and $L_T(s)+L_T(t)+L_T(r)$, respectively. Concerning trichotomy for $s=\xi\langle a,\sigma\rangle$ and $t=\xi\langle b,\tau\rangle$, we use the induction hypothesis to infer that $<_{\fix(T)}$ is linear on $a\cup b$ (note that $r<_{\fix(T)}r'<_{\fix(T)}r\rightarrow r<_{\fix(T)}r$ is available for $r,r'\in a\cup b$, due to the factor $2$ in the definition of $L_T$). According to the previous definition we obtain an inequality between $s$ and $t$, unless we have
\begin{equation*}
T(|\iota_a^{a\cup b}|)(\sigma)=T(|\iota_b^{a\cup b}|)(\tau).
\end{equation*}
Due to $\supp^T_{|a|}(\sigma)=|a|$ (cf.~Definition~\ref{def:fix-T}) and the naturality of $\supp^T$ we see that~$a$ can be recovered from the left side of this equality, namely as
\begin{multline*}
a=[\iota_a^{a\cup b}]^{<\omega}\circ[\en_a]^{<\omega}(\supp^T_{|a|}(\sigma))=[\en_{a\cup b}]^{<\omega}\circ[|\iota_a^{a\cup b}|]^{<\omega}(\supp^T_{|a|}(\sigma))
=\\
=[\en_{a\cup b}]^{<\omega}(\supp^T_{|a\cup b|}(T(|\iota_a^{a\cup b}|)(\sigma))).
\end{multline*}
Since $b$ can be recovered in the same way, the above equality implies $a=b$. We can conclude that $|\iota_a^{a\cup b}|$ and $|\iota_b^{a\cup b}|$ coincide (in fact, they are both equal to the identity on $|a|$). Since the order embedding $T(|\iota_a^{a\cup b}|)=T(|\iota_b^{a\cup b}|)$ is injective we get $\sigma=\tau$ and thus $s=t$, as required for trichotomy. To establish transitivity between $s=\xi\langle a,\sigma\rangle$, $t=\xi\langle b,\tau\rangle$ and $r=\xi\langle c,\rho\rangle$ one considers the inclusions into $a\cup b\cup c$ and uses the fact that $T(|a\cup b\cup c|)$ is a linear order.
\end{proof}

We will see that the following yields a fixed point in the sense of Definition~\ref{def:fixed-point}.

\begin{definition}[$\rca_0$]\label{def:fix-fixed-point}
 We define $\xi_T:D^T(\fix(T))\rightarrow\fix(T)$ by stipulating
 \begin{equation*}
  \xi_T(\langle a,\sigma\rangle)=\xi\langle a,\sigma\rangle,
 \end{equation*}
 for each normal prae-dilator~$T$.
\end{definition}

Let us verify the expected property:

\begin{proposition}[$\rca_0$]\label{prop:fix-T-fixed-point}
 Given a normal prae-dilator~$T$, the order $\fix(T)$ and the function $\xi_T:D^T(\fix(T))\rightarrow\fix(T)$ form a fixed point of $T$.
\end{proposition}
\begin{proof}
 According to Definition~\ref{def:fixed-point} we must show that $\xi_T$ is an order embedding. In view of Definitions~\ref{def:coded-prae-dilator-reconstruct} and~\ref{def:fix-T-order} the implication
 \begin{equation*}
  \langle a,\sigma\rangle <_{D^T(\fix(T))}\langle b,\tau\rangle\quad\Rightarrow\quad \xi\langle a,\sigma\rangle <_{\fix(T)}\xi\langle b,\tau\rangle
 \end{equation*}
 is immediate, provided $<_{\fix(T)}$ is linear on $a\cup b$. The latter holds by Lemma~\ref{lem:fix-T-linear}.
\end{proof}

After Definition~\ref{def:fixed-point} we have discussed additional properties of fixed points, which one might want to require. Let us show that these properties are satisfied for the fixed point that we have constructed.

\begin{theorem}[$\rca_0$]\label{thm:fixed-point-additional}
 The following holds for any normal prae-dilator~$T$:
 \begin{enumerate}[label=(\alph*)]
  \item The embedding $\xi_T:D^T(\fix(T))\rightarrow\fix(T)$ is an isomorphism.
  \item Given any fixed point $\xi_X:D^T(X)\rightarrow X$ of $T$, there is an order embedding of~$\fix(T)$ into~$X$.
 \end{enumerate}
\end{theorem}
\begin{proof}
 In view of Definitions~\ref{def:fix-T} and~\ref{def:coded-prae-dilator-reconstruct} it is clear that $\xi\langle a,\sigma\rangle\in\fix(T)$ implies $\langle a,\sigma\rangle\in D^T(\fix(T))$, which yields claim~(a). To establish claim~(b) we construct a function $j:\fix(T)\rightarrow X$ by recursion over terms, setting
 \begin{equation*}
  j(\xi\langle a,\sigma\rangle)=\xi_X(\langle[j]^{<\omega}(a),\sigma\rangle).
 \end{equation*}
 By simultaneous induction on $L_T(r)$ and $L_T(s)+L_T(t)$, respectively, we show
 \begin{align*}
  r\in\fix(T)&\rightarrow j(r)\in X,\\
  s<_{\fix(T)}t&\rightarrow j(s)<_X j(t).
 \end{align*}
 To establish the first claim we write $r=\xi\langle a,\sigma\rangle$. The simultaneous induction hypothesis implies that $j$ is strictly increasing and hence injective on~$a$. Given that~$a$ and~$[j]^{<\omega}(a)$ have the cardinality, it is immediate that \mbox{$\xi\langle a,\sigma\rangle\in\fix(T)$} implies $\langle [j]^{<\omega}(a),\sigma\rangle\in D^T(X)$, as needed. To show the second claim we assume
 \begin{equation*}
  s=\xi\langle a,\sigma\rangle<_{\fix(T)}\xi\langle b,\tau\rangle=t.
 \end{equation*}
 According to Definition~\ref{def:fix-T-order} this inequality amounts to
 \begin{equation*}
  T(|\iota_a^{a\cup b}|)(\sigma)<_{T(|a\cup b|)} T(|\iota_b^{a\cup b}|)(\tau)
 \end{equation*}
 The induction hypothesis tells us that $j$ is order preserving on $a\cup b$. This yields
 \begin{equation*}
  j\circ\en_{a\cup b}=\en_{[j]^{<\omega}(a\cup b)},
 \end{equation*}
 since the functions on both sides enumerate the set $[j]^{<\omega}(a\cup b)$ in increasing order (with respect to $<_X$). Using the defining property of $|\iota_a^{a\cup b}|$ we obtain
 \begin{multline*}
  \en_{[j]^{<\omega}(a\cup b)}\circ|\iota_a^{a\cup b}|=j\circ\en_{a\cup b}\circ|\iota_a^{a\cup b}|=j\circ\iota_a^{a\cup b}\circ\en_a=\\
  =\iota_{[j]^{<\omega}(a)}^{[j]^{<\omega}(a\cup b)}\circ j\circ\en_a=\iota_{[j]^{<\omega}(a)}^{[j]^{<\omega}(a\cup b)}\circ \en_{[j]^{<\omega}(a)},
 \end{multline*}
 where the last equality is established as above. Since the functions $|f|$ are uniquely determined by their defining property, we can conclude
 \begin{equation*}
 \left|\iota_{[j]^{<\omega}(a)}^{[j]^{<\omega}(a\cup b)}\right|=\left|\iota_a^{a\cup b}\right|.
 \end{equation*}
 The same holds with $b$ at the place of $a$. Hence we get
 \begin{equation*}
 T\left(\left|\iota_{[j]^{<\omega}(a)}^{[j]^{<\omega}(a\cup b)}\right|\right)(\sigma)<_{T(|a\cup b|)} T\left(\left|\iota_{[j]^{<\omega}(b)}^{[j]^{<\omega}(a\cup b)}\right|\right)(\tau).
 \end{equation*}
 In view of Definition~\ref{def:fix-T-order} this yields $\langle[j]^{<\omega}(a),\sigma\rangle<_{D^T(X)}\langle[j]^{<\omega}(b),\tau\rangle$. Since $\xi_X$ is an order embedding we can infer
 \begin{equation*}
  j(s)=\xi_X(\langle[j]^{<\omega}(a),\sigma\rangle)<_X\xi_X(\langle[j]^{<\omega}(b),\tau\rangle)=j(t),
 \end{equation*}
 as required.
\end{proof}

Let us point out that the function $j$ that we have constructed in the previous proof respects the structure of the fixed points $(\fix(T),\xi_T)$ and $(X,\xi_X)$. To see what this means we recall that $j$ induces a function $D^T(j):D^T(\fix(T))\rightarrow D^T(X)$, given by $D^T(j)(\langle a,\sigma\rangle)=\langle[j]^{<\omega}(a),\sigma\rangle$ (cf.~the discussion after Definition~\ref{def:coded-dilator}). Hence the defining equation of $j$ amounts to
\begin{equation*}
 j\circ\xi_T=\xi_X\circ D^T(j).
\end{equation*}
An order embedding $j$ with this property could be called a morphism of fixed points. A straightforward induction on terms shows that all morphisms from $\fix(T)$ to $X$ must coincide. Hence $\fix(T)$ can be characterized as the initial fixed point of $T$, which is unique up to isomorphism.

In the first half of this section we have constructed a fixed point $\fix(T)$ of a given normal prae-dilator~$T$, working in $\rca_0$. To complete the proof of our main result we will now use $\Pi^1_1$-induction along the natural numbers to show that $\fix(T)$ is well-founded whenever $X\mapsto D^T(X)$ preserves well-foundedness (so that $T$ is a normal dilator). For this purpose we consider the construction of $\fix(T)$ in stages:

\begin{definition}[$\rca_0$]\label{def:term-height}
 Given a normal prae-dilator $T$, we define a height function $h_T:\fix(T)\rightarrow\mathbb N$ by setting
 \begin{equation*}
  h_T(\xi\langle a,\sigma\rangle)=\max(\{h_T(s)+1\,|\,s\in a\}\cup\{0\}).
 \end{equation*}
 For each number $n$ we consider the set
 \begin{equation*}
  \fix_n(T)=\{s\in\fix(T)\,|\, h_T(s)<n\},
 \end{equation*}
 ordered as a subset of $\fix(T)$.
\end{definition}

An infinite union of well-orders is not generally well-ordered. However, it is straightforward to see that an order is well-founded if it is the union of well-founded initial segments. This explains the importance of the following result, which is similar to Proposition~5.6 of~\cite{freund-rathjen_derivatives}. Note that the proof makes crucial use of the assumption that $T$ is normal.

\begin{proposition}[$\rca_0$]\label{prop:fix-initial-segment}
 Consider a normal prae-dilator~$T$. The order $\fix_n(T)$ is an initial segment of $\fix(T)$, for any number~$n$.
\end{proposition}
\begin{proof}
 It suffices to show that
 \begin{equation*}
  h_T(s)<h_T(t)\quad\Rightarrow\quad s<_{\fix(T)} t
 \end{equation*}
 holds for all $s,t\in\fix(T)$. Arguing by induction on $L_T(s)+L_T(t)$, we consider terms $s=\xi\langle a,\sigma\rangle$ and $t=\xi\langle b,\tau\rangle$. If we have $h_T(s)<h_T(t)$, then there must be an element $t'\in b$ such that $h_T(s')<h_T(t')$ holds for all $s'\in a$. By induction hypothesis we get $a\subseteq\fix(T)\!\restriction\!t'$. Also note that $s\in\fix(T)$ implies $\langle a,\sigma\rangle\in D^T(\fix(T))$. Since $T=(T,\mu^T)$ is normal we can invoke Proposition~\ref{prop:normal-dilator-extend} to obtain
 \begin{equation*}
  \langle a,\sigma\rangle<_{D^T(\fix(T))} D^{\mu^T}_{\fix(T)}(t').
 \end{equation*}
 On the other hand $t'\in b$ yields $b\not\subseteq \fix(T)\!\restriction\!t'$ and hence
 \begin{equation*}
  D^{\mu^T}_{\fix(T)}(t')\leq_{D^T(\fix(T))}\langle b,\tau\rangle.
 \end{equation*}
 From Proposition~\ref{prop:fix-T-fixed-point} we know that $\xi_T$ is order preserving. We can thus conclude
 \begin{equation*}
  s=\xi_T(\langle a,\sigma\rangle)<_{\fix(T)}\xi_T(\langle b,\tau\rangle)=t,
 \end{equation*}
 as required.
\end{proof}

In order to deduce the well-foundedness of $\fix_{n+1}(T)$ from the one of $\fix_n(T)$ we will use the following result:

\begin{proposition}[$\rca_0$]
 For any normal prae-dilator $T$ and any number $n$ we have an isomorphism $D^T(\fix_n(T))\cong\fix_{n+1}(T)$ of linear orders.
\end{proposition}
\begin{proof}
Proposition~\ref{prop:fix-T-fixed-point} tells us that $\xi_T:D^T(\fix(T))\rightarrow\fix(T)$ is an order embedding. Also note that
 \begin{equation*}
 D^T(\fix_n(T))=\{\langle a,\sigma\rangle\in D^T(\fix(T))\,|\,a\subseteq\fix_n(T)\}
 \end{equation*}
 is a suborder of $D^T(\fix(T))$ (cf.~the proof of Theorem~\ref{thm:normal-dil-fct}). To conclude it suffices to show that $\xi_T$ maps $D^T(\fix_n(T))$ onto $\fix_{n+1}(T)$. Given $a\subseteq\fix_n(T)$, we observe that $h_T(s)<n$ holds for all $s\in a$. This implies
 \begin{equation*}
 h_T(\xi_T(\langle a,\sigma\rangle))=h_T(\xi\langle a,\sigma\rangle)=\sup\{h_T(s)+1\,|\,s\in a\}\leq n<n+1,
 \end{equation*}
 as required for $\xi_T(\langle a,\sigma\rangle)\in\fix_{n+1}(T)$. Conversely, Theorem~\ref{thm:fixed-point-additional} shows that any element of $\fix_{n+1}(T)$ arises as the image $\xi_T(\langle a,\sigma\rangle)$ of some $\langle a,\sigma\rangle\in D^T(\fix(T))$. As above we see that $h_T(\xi_T(\langle a,\sigma\rangle))<n+1$ implies $h_T(s)+1<n+1$ and hence $h_T(s)<n$ for all $s\in a$, so that we get $a\subseteq\fix_n(T)$.
\end{proof}

Putting things together, we can now prove the main result of our paper:

\begin{theorem}[$\aca_0$]\label{thm:main-result}
The following are equivalent:
\begin{enumerate}[label=(\roman*)]
\item Every coded normal dilator has a well-founded fixed point.
\item If $T$ is a coded normal dilator, then $\fix(T)$ is well-founded.
\item The principle of $\Pi^1_1$-induction along the natural numbers holds.
\end{enumerate}
\end{theorem}
\begin{proof}
From Theorem~\ref{thm:fixed-to-ind} we know that (i) implies (iii). Let us point out that the proof of this direction uses arithmetical comprehension, in the form of the Kleene normal form theorem and the well-foundedness of the Kleene-Brouwer order on a tree without infinite branch. The other directions can be established over $\rca_0$: To see that (ii) implies (i) it suffices to recall that $\fix(T)$ is a fixed point of the given dilator $T$, due to Proposition~\ref{prop:fix-T-fixed-point}. It remains to show that (iii) implies (ii). For this purpose we consider a normal dilator~$T$. According to Proposition~\ref{prop:fix-initial-segment} the order $\fix(T)$ can be written as a union
\begin{equation*}
\fix(T)=\bigcup_{n\in\mathbb N}\fix_n(T)
\end{equation*}
of initial segments. Thus the well-foundedness of $\fix(T)$ reduces to the claim that $\fix_n(T)$ is well-founded for every number~$n$. To establish the latter we argue by induction on~$n$, as justified by~(iii). In view of $\fix_0(T)=\emptyset$ the base case $n=0$ is trivial. Now assume that $\fix_n(T)$ is well-founded. Since $T$ is a dilator this implies the well-foundedness of $D^T(\fix_n(T))$ (cf.~Definition~\ref{def:coded-dilator}). Using the previous proposition we can infer that $\fix_{n+1}(T)\cong D^T(\fix_n(T))$ is well-founded, as required for the induction step.
\end{proof}

The author would like to thank the referee for pointing out the following:

\begin{remark}\label{rmk:large-fixed-points}
The statements from Theorem~\ref{thm:main-result} are also equivalent to the following, still over $\aca_0$:
\begin{enumerate}[label=(\roman*)]\setcounter{enumi}{3}
\item If $T$ is a coded normal dilator, then any well-order~$X$ can be embedded into some well-founded fixed point of~$T$.
\end{enumerate}
It is immediate that (iv) implies~(i). Conversely, we will establish~(iv) by applying~(i) to a modified dilator~$T[X]$. Given a finite order $n=\{0,\dots,n-1\}$, we consider the disjoint union
\begin{equation*}
T[X](n)=X+T(n).
\end{equation*}
To obtain a linear order, we declare that $x<_{T[X](n)}y<_{T[X](n)}\sigma<_{T[X](n)}\tau$ holds for any elements $x<_X y$ of~$X$ and $\sigma<_{T(n)}\tau$ of $T(n)$. Given a morphism $f:n\to m$, we define $T[X](f):T[X](m)\to T[X](n)$ by setting
\begin{equation*}
T[X](f)(\sigma)=\begin{cases}
\sigma & \text{if $\sigma\in X\subseteq T[X](m)$,}\\
T(f)(\sigma)   & \text{if $\sigma\in T(m)\subseteq T[X](m)$.}\\
\end{cases}
\end{equation*}
To obtain a coded prae-dilator, we also define $\supp^{T[X]}_n:T[X](n)\to[n]^{<\omega}$ by
\begin{equation*}
\supp^{T[X]}_n(\sigma)=\begin{cases}
\emptyset & \text{if $\sigma\in X\subseteq T[X](n)$,}\\
\supp^{T}_n(\sigma)   & \text{if $\sigma\in T(n)\subseteq T[X](n)$.}\\
\end{cases}
\end{equation*}
As a normal dilator, $T$~comes with a natural family of embeddings $\mu^T_n:n\to T(n)$. In order to turn $T[X]$ into a coded normal prae-dilator we set
\begin{equation*}
\mu^{T[X]}_n(m)=\mu^T_n(m)\in T(n)\subseteq T[X](n).
\end{equation*}
The condition from Definition~\ref{def:normal-dil} is preserved as both $x<_{T[X](n)}\mu^{T[X]}_n(m)\in T(n)$ and $\supp^{T[X]}_n(x)=\emptyset\subseteq m$ is true for any $x\in X$. For any order~$Y$ we have an order isomorphism
\begin{equation*}
X+D^T(Y)\cong D^{T[X]}(Y),
\end{equation*}
where $x\in X$ corresponds to $\langle\emptyset,x\rangle\in D^{T[X]}(Y)$. Given that $T$ is a dilator and $X$ is a well-order, it follows that $Y\mapsto D^{T[X]}(Y)$ preserves well-foundedness. This means that $D^{T[X]}$ is a coded normal dilator. Hence statement~(i) from Theorem~\ref{thm:main-result} yields a well-order~$Y$ that admits an embedding
\begin{equation*}
\xi:D^{T[X]}(Y)\rightarrow Y.
\end{equation*}
In view of $X+D^T(Y)\cong D^{T[X]}(Y)$ we obtain an embedding of $D^T(Y)$ into~$Y$, which shows that $Y$ is also a fixed point of $T$ (but here we do not get $D^T(Y)\cong Y$, in contrast to Theorem~\ref{thm:fixed-point-additional}). To establish~(iv) it remains to embed $X$ into~$Y$. For this purpose we compose~$\xi$ with the inclusion of $X$ into $X+D^T(Y)\cong D^{T[X]}(Y)$.
\end{remark}

\bibliographystyle{amsplain}
\bibliography{Single_Fixed-point}

\end{document}